%% file: main.tex
\documentclass[journal]{IEEEtran}

\usepackage{cite}
\usepackage{amsmath}
\usepackage{amsthm}
\usepackage{amsfonts}
\usepackage[pdfstartview=XYZ,
bookmarks=true,
colorlinks=true,
linkcolor=blue,
urlcolor=blue,
citecolor=blue,
pdftex,
bookmarks=true,
linktocpage=true, % makes the page number as hyperlink in table of content
hyperindex=true
]{hyperref}
\usepackage{orcidlink}
\newtheorem{theorem}{Theorem}
\newtheorem{lemma}{Lemma}
\theoremstyle{definition}
\newtheorem{corollary}{Corollary}
\newtheorem{example}{Example}
\usepackage{amssymb}
\usepackage{fancyhdr}
\theoremstyle{plain}

\usepackage{enumitem}
\newcounter{Ax}
\newcommand{\itemA}{%
    \addtocounter{Ax}{1}
    \item[(A\theAx)]}

\fancypagestyle{firststyle}
{
   \fancyhf{}
    \fancyhead[C]{\textcolor{blue}{Accepted version of the article published in IEEE Signal Processing Letters. DOI:10.1109/LSP.2024.3396660. Available at: http://ieeexplore.ieee.org}}
   \vspace{-1em} \fancyfoot[C]{\textcolor{blue}{\scriptsize \textcopyright 2024 IEEE. Personal use of this material is permitted. Permission from IEEE must be obtained for all other uses, in any current or future media, including reprinting/republishing this material for advertising or promotional purposes, creating new collective works, for resale or redistribution to servers or lists, or reuse of any copyrighted component of this work in other works.}
   }
}

\begin{document}

\title{On the convergence of Block Majorization-Minimization algorithms on the Grassmann Manifold}

\author{Carlos~Alejandro~Lopez \orcidlink{0000-0002-2216-2786},~\IEEEmembership{Student Member,~IEEE,} and
        Jaume~Riba \orcidlink{0000-0002-5515-8169},~\IEEEmembership{Senior Member,~IEEE}% 
\thanks{
    This work was supported by project MAYTE (PID2022-136512OB-C21 financed by MCIN/AEI/10.13039/501100011033 and by "ERDF A way of making Europe", EU), by project RODIN (PID2019-105717RB-C22/AEI/10.13039/501100011033), by the grant 2021 SGR 01033, and the fellowship 2023 FI-3 00155 by Generalitat de Catalunya and the European Social Fund.}
    \thanks{The authors are with the Signal Processing and Communications Group (SPCOM), Dept. de Teoria del Senyal i Comunicacions, Universitat Polit{\`e}cnica de Catalunya (UPC), 08034 Barcelona, Spain (e-mail: carlos.alejandro.lopez@upc.edu; jaume.riba@upc.edu).}
    }

% note the % following the last \IEEEmembership and also \thanks - 
% these prevent an unwanted space from occurring between the last author name
% and the end of the author line. i.e., if you had this:
% 
% \author{....lastname \thanks{...} \thanks{...} }
%                     ^------------^------------^----Do not want these spaces!
%
% a space would be appended to the last name and could cause every name on that
% line to be shifted left slightly. This is one of those "LaTeX things". For
% instance, "\textbf{A} \textbf{B}" will typeset as "A B" not "AB". To get
% "AB" then you have to do: "\textbf{A}\textbf{B}"
% \thanks is no different in this regard, so shield the last } of each \thanks
% that ends a line with a % and do not let a space in before the next \thanks.
% Spaces after \IEEEmembership other than the last one are OK (and needed) as
% you are supposed to have spaces between the names. For what it is worth,
% this is a minor point as most people would not even notice if the said evil
% space somehow managed to creep in.

% The paper headers
\markboth{IEEE SIGNAL PROCESSING LETTERS, 2024}%
{Lopez \MakeLowercase{\textit{et al.}}: On the convergence of Block Majorization-Minimization algorithms on the Grassmann Manifold}

\maketitle
\thispagestyle{firststyle}

\begin{abstract}
    \input{Sections/abstract}
\end{abstract}


\IEEEpeerreviewmaketitle

    \section{Introduction}
        \input{Sections/intro}

    \section{Mathematical Preliminaries and Definitions}
        \input{Sections/preliminaries}

    \section{Block MM convergence proof for variables constrained in the Grassmann Manifold}
        \input{Sections/proof}

    \section{Conclusions}
        \input{Sections/conclusions}

\end{document}

%% file: Sections/abstract.tex
The Majorization-Minimization (MM) framework is widely used to derive efficient algorithms for specific problems that require the optimization of a cost function (which can be convex or not). It is based on a sequential optimization of a surrogate function over closed convex sets. A natural extension of this framework incorporates ideas of Block Coordinate Descent (BCD) algorithms into the MM framework, also known as block MM. The rationale behind the block extension is to partition the optimization variables into several independent blocks, to obtain a surrogate for each block, and to optimize the surrogate of each block cyclically. However, known convergence proofs of the block MM are only valid under the assumption that the constraint sets are closed and convex. Hence, the global convergence of the block MM is not ensured for non-convex sets by classical proofs, which is needed in iterative schemes that naturally emerge in a wide range of subspace-based signal processing applications. For this purpose, the aim of this letter is to review the convergence proof of the block MM and extend it for blocks constrained in the Grassmann manifold.

%% file: Sections/intro.tex
\IEEEPARstart{T}{h} Majorization-Minimization (MM) framework is an optimization paradigm that is capable to meet the requirements of emerging applications \cite{big_data_review}, which often necessitate that the algorithms are capable to handle multimodal datasets, provide integrity to the final solution, be computationally fast, and scale efficiently. This framework is based on the iterative minimization of a surrogate function (also referred to as the majorant \cite{expanded_theoretical_treatment_mm}) of the original cost. This surrogate is a tight upperbound of the cost function at the current iterate and it is often selected such that the variables are separable (for parallel computing) and that it is easy to find an optimal solution of the surrogate problem \cite{mm_review_palomar}. The MM framework is often considered to be a generalization of other well-known algorithmic frameworks such as the Expectation-Maximization \cite{MM_alternative_EM, EM_vs_MM}, the Concave-Convex Procedure \cite{concave-convex-procedure}, and the Proximal algorithms \cite{proximal_algorithms}, among others.

In relation to the separability of the optimization variables, a common approach is to combine the Block Coordinate Descent (BCD) procedure \cite{block_coordinate_descent} with the MM framework. This approach, also termed as block MM in the literature, consists on the partition of the variables into independent blocks with the aim of employing the MM framework in each block. The motivation behind the block MM framework is that, in some cases, the original cost function can be better approximated using several blocks \cite{mm_review_palomar}. However, classical convergence theorems of block MM algorithms assume that each block is constrained in a closed convex set \cite{unified_convergence_analysis, expanded_theoretical_treatment_mm} and thus they do not ensure the global convergence for non-convex sets. 

Examples of widely used non-convex sets in optimization are the Stiefel and Grassmann manifolds \cite{orth_grassmann}. Those manifolds appear in problems with orthogonality constraints, which are commonly faced in signal processing applications, such as Principal Component Analysis (PCA) \cite{PCA_FA, robust_pca_partial_subspace}, Subspace Learning \cite{grouse, apst, fast_stable_st}, Matrix Factorization \cite{orthogonal_nmf, riemannian_perspective_matrix} and Non-Coherent Communications \cite{non_coherent_comms_grassmann}, to name a few. 

Although the convergence of the MM framework on Stiefel (tightly related to the Grassmann manifold) constrained variables is proven in \cite{mm_stiefel}, no rigorous proofs exist for the block MM case. For this reason, the main goal of this letter is to revisit the convergence proof of the block MM in \cite{unified_convergence_analysis} and to extend it for blocks constrained in the Grassmann manifold. Despite the non-convexity of the Grassmanian, we show that the compactness, the geodesic convexity of subsets and the geodesic quasiconvexity \cite{geodesically_quasiconvex} of the surrogate functions are sufficient conditions to ensure the convergence of a block MM algorithm involving the Grassmannian. Examples of the theoretical need for this convergence proof are found in any algorithm that combines PCA with the estimation of additional parameters or in the blind deconvolution problem \cite{dimensionality_reduction_binary_data_natural_parameters, Truncated_power_method_sparse_eigenvalue, Nonconvex_approach_exact_efficient_multichannel, Dual_principal_component_pursuit, global_geometry_sphere_constrained_sparse, short_sparse_deconvolution}. Additionally, the theoretical study in this paper serves as an alternative way to proof the convergence of the Proximal Gradient method on the Grassmann manifold (see \cite{proximal_gradient_stiefel_manifold} for more insights).

%% file: Sections/preliminaries.tex
We next review some of the necessary concepts involved in the convergence proof without the necessity to delve into differential geometry. We refer to \cite{orth_grassmann} for an in-depth treatment of the Grassmann manifold geometry for optimization algorithms and to \cite{convex_geodesic_overview, accelerated_geodesically_convex_opt} for an overview on geodesically convex optimization.

\subsection{Grassmann manifold definitions and concepts}

\subsubsection{Grassmann manifold}
\label{subsubsec:Grassmann_manifold}

The Grassmann manifold, $\text{Gr}(N, D)$, is the set of $D$-dimensional subspaces in $\mathbb{R}^N$. We consider that each point is represented by the following equivalence class:
\begin{equation}
    \label{eqn:eq_class}
    [\mathbf{X}] = \{ \mathbf{X}\mathbf{R} \in \mathbb{R}^{N \times D}: \mathbf{X}^T\mathbf{X} = \mathbf{I}_D, \mathbf{R} \in \text{O}(D) \},
\end{equation}
where $\mathbf{X} \in \mathbb{R}^{N \times D}$ and $\text{O}(D)$ denotes the set of matrices that satisfy $\mathbf{R}^T\mathbf{R} = \mathbf{R}\mathbf{R}^T = \mathbf{I}_D$. We identify the entire equivalent class with a particular representative, $\mathbf{X}$, while a subindex is used to represent any related matrix that spans the same subspace, e.g. $\mathbf{X}_r= \mathbf{X}\mathbf{R}$. This approach of representing points in the Grassmann manifold is often referred to as the Orthonormal Basis (ONB) perspective \cite{grassmann_handbook}. In other words, we represent a point in the Grassmann manifold with a variable constrained in the Stiefel manifold \cite{orth_grassmann}.

\subsubsection{Tangent Space of the Grassmann manifold}

Let $\mathbf{X} \in \text{Gr}(N, D)$, then:
\begin{equation}
    \mathcal{T}_{\mathbf{X}} \text{Gr}(N, D) = \{ \boldsymbol{\Delta} \in \mathbb{R}^{N \times D} :  \mathbf{X}^T\boldsymbol{\Delta} = \mathbf{0}\},
\end{equation}
denotes the tangent space at $\mathbf{X}$.

\subsubsection{Principal angles between two subspaces}

In $\text{Gr}(N, D)$, the principal angles between any two points $[\mathbf{X}]$ and $[\mathbf{Y}]$ are defined as the minimal angles between all possible subspace basis of those two points \cite{grassmannian_learning}. They can be computed efficiently using the Singular Value Decomposition (SVD) of $\mathbf{X}^T\mathbf{Y}$. Indeed, the SVD of $\mathbf{X}^T\mathbf{Y}$ can be expressed as $\mathbf{U} \cos(\boldsymbol{\Theta}) \mathbf{V}^T$ where $\boldsymbol{\Theta}$ is the diagonal matrix containing the $D$ principal angles and $\cos (\cdot)$ is the cosine function applied element-wise on the main diagonal of its input matrix. Provided that the singular values are all positive, the matrix containing the principal angles is such that $[\boldsymbol{\Theta}]_{d, d} \in  [0, \frac{\pi}{2}]$ for all $d=1,...,D$ where $[ \cdot ]_{i,j}$ is the $(i, j)$-th entry of its input matrix. From the aforementioned SVD, one can obtain clean expressions using the following lemma (see \cite[Proposition 1]{riemannian_perspective_matrix} for more details):
\begin{lemma}
    \label{lem:princ}
    For any two points $[\mathbf{X}], [\mathbf{Y}] \in \text{Gr}(N, D)$, one can find two aligned representatives $\mathbf{X}_a$ and $\mathbf{Y}_a$ such that $\mathbf{X}^T_a\mathbf{Y}_a = \cos(\boldsymbol{\Theta})$.
\end{lemma}

\subsubsection{Grassmann geodesics}

A geodesic in the Grassmann manifold is defined as the shortest path between two subspaces. In fact, it is only when the principal angles between $[\mathbf{X}]$ and $[\mathbf{Y}]$ are all (strictly) smaller than $\frac{\pi}{2}$ that the geodesic that joins them is unique \cite{grassmann_unique_geodesic}. Consider the direction vector $\mathbf{H} \in \mathcal{T}_{\mathbf{X}} \text{Gr}(N, D)$ pointing to $\mathbf{Y}$ and its compact SVD, $\mathbf{H} = \mathbf{U}_h\boldsymbol{\Theta}\mathbf{V}^T_h$, such that $\boldsymbol{\Theta} \preceq \frac{\pi}{2} \mathbf{I}$ where $\boldsymbol{\Theta}$ is the matrix containing the principal angles between $[\mathbf{X}]$ and $[\mathbf{Y}]$. Then, the expression of the geodesic that connects $\mathbf{X}$ with $\mathbf{Y}$ is:
\begin{equation}
    \label{eqn:basic_eq_geodesic}
    \boldsymbol{\Gamma}(t) = \mathbf{X}\mathbf{V}_h\cos(\boldsymbol{\Theta}t)\mathbf{V}^T_h + \mathbf{U}_h\sin(\boldsymbol{\Theta}t)\mathbf{V}^T_h,
\end{equation}
where $\sin(\cdot)$ is applied element-wise on its input matrix main diagonal and $t \in [0, 1]$ parameterizes the geodesic. Note that the previous geodesic must be such that $\boldsymbol{\Gamma}(0) = \mathbf{X}$, $\boldsymbol{\Gamma}(1) = \mathbf{Y}$ and $\frac{\text{d}\boldsymbol{\Gamma}(t)}{\text{d}t} \big|_{t=0} = \mathbf{H}$. Given Lemma \ref{lem:princ}, there exists an orthonormal matrix $\boldsymbol{\Delta}_a$ such that the geodesic connecting the aligned representatives, $\mathbf{X}_a$ and $\mathbf{Y}_a$, yields:
\begin{equation}
    \label{eqn:geo_aligned}
    \boldsymbol{\Gamma}_a(t) = \mathbf{X}_a\cos(\boldsymbol{\Theta}t) + \boldsymbol{\Delta}_a\sin(\boldsymbol{\Theta}t),
\end{equation}
where $\boldsymbol{\Delta}_a$ is such that $\boldsymbol{\Gamma}_a(1) = \mathbf{Y}_a$. The intuition behind (\ref{eqn:geo_aligned}) can be grasped from the following expression:
\begin{equation}
    \mathbf{X}^T_a\mathbf{Y}_a = \mathbf{X}^T_a\boldsymbol{\Gamma}_a(1) = \cos(\boldsymbol{\Theta}) + \mathbf{X}^T_a\boldsymbol{\Delta}_a\sin(\boldsymbol{\Theta}),
\end{equation}
from where $\boldsymbol{\Delta}_a$ must comply with the following constraint so that Lemma \ref{lem:princ} holds:
\begin{equation}
    \mathbf{X}^T_a\boldsymbol{\Delta}_a = \mathbf{0},
\end{equation}
which means that $\boldsymbol{\Delta}_a$ belongs to $\mathcal{T}_{\mathbf{X}} \text{Gr}(N, D)$.

\subsubsection{Canonical distance of the Grassmann manifold}

The canonical distance between any two points $[\mathbf{X}], [\mathbf{Y}] \in \text{Gr}(N, D)$, also known as the arclength, is given by the following expression:
\begin{equation}
    d_c(\mathbf{X}, \mathbf{Y}) = || \boldsymbol{\Theta}||_F,
\end{equation}
where $|| \cdot ||_F$ is the Frobenius norm and $\boldsymbol{\Theta} \preceq \frac{\pi}{2} \mathbf{I}$ are the principal angles between $[\mathbf{X}]$ and $[\mathbf{Y}]$.

\subsection{Geodesically convex optimization concepts}

\subsubsection{Geodesically convex set}

A set $\mathcal{G} \subseteq \text{Gr}(N, D)$ is geodesically convex if there exists a geodesic joining any two points of $\mathcal{G}$ that lies entirely on $\mathcal{G}$. An example of a geodesically convex subset of the Grassmann manifold can be found in \cite{riemannian_perspective_matrix}. This concept is a particularization of totally convex sets \cite{convex_geodesic_overview}.

\subsubsection{Geodesic quasi-convexity}

A function $f(\mathbf{X})$ is geodesically quasiconvex in the geodesically convex set $\mathcal{G} \subseteq \text{Gr}(N, D)$ with respect to the geodesic $\boldsymbol{\Gamma}(t)$ if \cite{geodesically_quasiconvex}:
\begin{equation}
    \label{eqn:quasiconvexity_definition}
    f(\boldsymbol{\Gamma}(t)) \leq \max(f(\mathbf{X}), f(\mathbf{Y}))\,\,\forall \mathbf{X}, \mathbf{Y} \in \mathcal{G},
\end{equation}
for $t \in [0, 1]$. Note that (\ref{eqn:quasiconvexity_definition}) is a particularization on Riemannian Manifolds of the definition of connected functions in \cite{expanded_theoretical_treatment_mm}. Convexity implies quasiconvexity:
\begin{equation}
   (1-t)f(\mathbf{X}) + t f(\mathbf{Y}) \leq \max(f(\mathbf{X}), f(\mathbf{Y}))\,\, \forall t\in [0, 1].
\end{equation}

\subsubsection{Lower directional derivative}

For any function $f(\mathbf{X})$, the lower directional derivative at a point $\mathbf{X}$ in the direction $\boldsymbol{\Delta}$ is defined as:
\begin{equation}
    f'(\mathbf{X}; \boldsymbol{\Delta}) = \lim_{h \xrightarrow{} 0} \inf \frac{f(\mathbf{X} + h \boldsymbol{\Delta}) - f(\mathbf{X})}{h},
\end{equation}
which is an alternative definition of a derivative that encompasses a wider range of functions than the classical definition of derivative. This definition is independent of whether $\mathbf{X}$ is constrained to the Grassmann manifold or not.

\subsubsection{Regularity of a function}

Let the partition of $\mathbf{X}$ into $N$ blocks of variables be $\mathbf{X} = (\mathbf{X}_1, \mathbf{X}_2,...,\mathbf{X}_N)$ where the dimensions of the $n$-th block of variables are denoted by $N_n$ for $n=1,...,N$. In addition, we denote $\boldsymbol{\Delta}_n$ as the vector that contains all zeroes except in the $n$-th block containing the direction $\mathbf{D}_n$ (of dimensions $N_n$), i.e. $\boldsymbol{\Delta}_n = (\mathbf{0},...,\mathbf{D}_n,...,\mathbf{0})$. Then, $f(\mathbf{X})$ is regular in $\mathcal{X}$ if:
\begin{equation}
    \label{eqn:regular_cond}
    \begin{gathered}
        f'(\mathbf{X}; \boldsymbol{\Delta}) \geq 0\,\,\forall \mathbf{X} \in \mathcal{X}, \forall \boldsymbol{\Delta} \\
        \text{and} \,\,\,\,f'(\mathbf{X}; \boldsymbol{\Delta}_n) \geq 0 \,\,\forall \mathbf{X} \in \mathcal{X}, \forall \boldsymbol{\Delta}_n, \,\,  n=1,...,N,
    \end{gathered}
\end{equation}
where $\boldsymbol{\Delta}$ and $\boldsymbol{\Delta}_n$ are such that any path emanating from $\mathbf{X}$ in those directions remains on $\mathcal{X}$. In other words, every coordinate-wise local minimum of $f(\mathbf{X})$ is also a stationary point. A simple case in which a function is regular is a differentiable function, but the regularity of a function can also be proven for some non-differentiable functions as well. We refer to \cite[Lemma 3.1]{convergence_bcd} for more details on this definition of regularity.

%% file: Sections/proof.tex
Consider the following optimization problem:
\begin{equation}
    \label{eqn:prob_stat}
    \hat{\mathbf{X}} = \arg \min_{\mathbf{X}} f(\mathbf{X}) \,\,\,\,\text{s.t.}\,\,\mathbf{X} \in \mathcal{X},
\end{equation}
where $f(\mathbf{X})$ is any continuous function, $\mathcal{X} = \mathcal{G} \times \mathcal{C}$, $\mathcal{C}$ is any closed convex set and $\mathcal{G} \subseteq \text{Gr}(N, D)$ is any geodesically convex subset of the Grassmann manifold. The reason of choosing only two blocks, one constrained in the Grassmannian and the other being a classical convex constraint set, is to emphasize the novelty of our convergence proof (block MM with Grassmann blocks) without any loss of generality. 

Given the structure of the constraint set, $\mathcal{X}$, the optimization variables are split into two independent blocks, $\mathbf{X} = (\mathbf{G}, \mathbf{c})$ where $\mathbf{G} \in \mathcal{G}$ and $\mathbf{c} \in \mathcal{C}$. Accordingly, we rewrite $f(\mathbf{X})$ as $f(\mathbf{G}, \mathbf{c})$ from now on. Note that since the optimization variable is a function of a Grassmann variable block, the cost function in (\ref{eqn:prob_stat}) must satisfy the following homogeneity condition \cite{orth_grassmann}:
\begin{equation}
    f(\mathbf{G}, \mathbf{c}) = f(\mathbf{G}\mathbf{R}, \mathbf{c}),
\end{equation}
for any $D \times D$ orthonormal matrix $\mathbf{R}$. The block MM rationale applied to (\ref{eqn:prob_stat}) consists on finding the solution (global minimum or stationary point) of (\ref{eqn:prob_stat}) by the successive minimization of the following problems:
\begin{subequations}
    \label{eqn:MM_update_general}
    \begin{equation}
        \label{eqn:update_G}
        \mathbf{G}_{i+1} = \arg \min_{\mathbf{G}} g_G(\mathbf{G} | \mathbf{G}_i, \mathbf{c}_i) \,\,\,\,\mathbf{G} \in \mathcal{G},
    \end{equation}
    \begin{equation}
        \label{eqn:update_c}
        \mathbf{c}_{i+1} = \arg \min_{\mathbf{c}} g_c(\mathbf{c} | \mathbf{G}_{i+1}, \mathbf{c}_i) \,\,\,\,\mathbf{c} \in \mathcal{C},
    \end{equation}
\end{subequations}
where $g_G(\mathbf{G} | \mathbf{G}_i, \mathbf{c}_i)$ and $g_c(\mathbf{c} | \mathbf{G}_i, \mathbf{c}_i)$ are the majorant functions of $f(\mathbf{G}, \mathbf{c})$ for each respective block constructed using the $i$-th iterates. As a result of the update equations in (\ref{eqn:MM_update_general}), a sequence of iterates is generated, denoted as $\{ \mathbf{G}_i, \mathbf{c}_i\}_{i \in \mathbb{N}}$. We assume the following properties of the majorant functions:
\begin{enumerate}
    \itemA The surrogates must have the same value as the original cost at the current iterate:
    \begin{equation*}
        \begin{gathered}
            g_G(\mathbf{G} | \mathbf{G}, \mathbf{c}) = f(\mathbf{G}, \mathbf{c}) \,\,\forall \mathbf{G} \in \mathcal{G}, \forall \mathbf{c} \in \mathcal{C}, \\
            g_c(\mathbf{c} | \mathbf{G}, \mathbf{c}) = f(\mathbf{G}, \mathbf{c}) \,\,\forall \mathbf{G} \in \mathcal{G}, \forall \mathbf{c} \in \mathcal{C}.
        \end{gathered}
    \end{equation*}
    \itemA The surrogates must majorize the original cost function:
    \begin{equation*}
        \begin{gathered}
            g_G(\mathbf{H} | \mathbf{G}, \mathbf{c}) \geq f(\mathbf{G}, \mathbf{c}) \,\,\forall \mathbf{G}, \mathbf{H} \in \mathcal{G}, \forall \mathbf{c} \in \mathcal{C}, \\
            g_c(\mathbf{d} | \mathbf{G}, \mathbf{c}) \geq f(\mathbf{G}, \mathbf{c}) \,\,\forall \mathbf{G} \in \mathcal{G}, \forall \mathbf{c},\mathbf{d} \in \mathcal{C}.
        \end{gathered}
    \end{equation*}
    \itemA The first (lower) directional derivatives of the surrogates and of the original cost must agree:
    \begin{equation*}
        g'_G(\mathbf{G} | \mathbf{G}, \mathbf{c}; \boldsymbol{\Delta}) = f'(\mathbf{G}, \mathbf{c}; \boldsymbol{\Delta}, \mathbf{0}),
    \end{equation*}
    for all tangent directions $\boldsymbol{\Delta} \in \mathcal{T}_{\mathbf{G}} \text{Gr}(N, D)$ whose resulting geodesic remains on $\mathcal{G}$ and:
    \begin{equation*}
        g'_c(\mathbf{c} | \mathbf{G}, \mathbf{c}; \boldsymbol{\delta}) = f'(\mathbf{G}, \mathbf{c}; \mathbf{0}, \boldsymbol{\delta}) \,\,\,\,\text{s.t.}\,\,\mathbf{c}+\boldsymbol{\delta} \in \mathcal{C}.
    \end{equation*}
    Note that $\mathbf{c}+\boldsymbol{\delta}$ is a geodesic in the Euclidean space.
    \itemA $g_G(\cdot | \cdot)$ and $g_c(\cdot | \cdot)$ must be continuous on its input arguments.
    \itemA $g_G(\mathbf{G} | \mathbf{G}, \mathbf{c})$ must be geodesically quasiconvex on $\mathcal{G}$ and $g_c(\mathbf{c} | \mathbf{G}, \mathbf{c})$ must be quasiconvex on $\mathcal{C}$.
\end{enumerate}

Notice that (A1) and (A2) ensure that the majorants are tight upperbounds of the original cost while the remaining assumptions enforce that the majorants resemble the original cost. In fact, we generalized the assumptions presented in \cite{unified_convergence_analysis} to equivalent expressions in terms of the geodesically convex optimization paradigm. We generalize Theorem 2 in \cite{unified_convergence_analysis} for blocks of variables constrained in the Grassmann manifold in the following theorem:
\begin{theorem}
    \label{theo:convergence}
    Supose that a sequence is generated by (\ref{eqn:MM_update_general}) and that the majorant functions satisfy (A1)-(A5). In addition, assume that (\ref{eqn:update_G}) and (\ref{eqn:update_c}) have unique minimizers, that the sublevel set, i.e. $S = \{\mathbf{X} \in \mathcal{X}: f(\mathbf{X}) \leq f(\mathbf{X}_0) \}$ for some initialization $\mathbf{X}_0$, is compact and that $f(\mathbf{X})$ is regular and continuous in $\mathcal{X}$. Then, this sequence converges to a stationary point of (\ref{eqn:prob_stat}).
\end{theorem}
\begin{proof}
    The rationale of the proof consists on two steps. Firstly, we prove that the sequence $\{ \mathbf{G}_i, \mathbf{c}_i\}_{i \in \mathbb{N}}$ converges to the limit point $\{ \bar{\mathbf{G}}, \bar{\mathbf{c}} \}$ and, secondly, we prove that this limit point is a stationary point of the problem depicted in (\ref{eqn:prob_stat}).

    From assumptions (A1)-(A2) and (\ref{eqn:MM_update_general}), we have the following series of inequalities:
    \begin{subequations}
        \label{eqn:sequence_f}
        \begin{equation}
            f(\mathbf{G}_i, \mathbf{c}_i) \underbrace{=}_{(A1)} g_G(\mathbf{G}_{i} | \mathbf{G}_i, \mathbf{c}_i) \underbrace{\geq}_{\text{Eq. } (\ref{eqn:update_G})} g_G(\mathbf{G}_{i+1} | \mathbf{G}_i, \mathbf{c}_i) \underbrace{\geq}_{(A2)}
        \end{equation}
        \begin{equation}
            f(\mathbf{G}_{i+1}, \mathbf{c}_i) \underbrace{=}_{(A1)} g_c(\mathbf{c}_i| \mathbf{G}_{i+1}, \mathbf{c}_i) \underbrace{\geq}_{\text{Eq. } (\ref{eqn:update_c})}
        \end{equation}
        \begin{equation}
            g_c(\mathbf{c}_{i+1}| \mathbf{G}_{i+1}, \mathbf{c}_i) \underbrace{\geq}_{(A2)} f(\mathbf{G}_{i+1}, \mathbf{c}_{i+1}),
        \end{equation}
    \end{subequations}
    where each underbrace describes the reason why each equality/inequality holds. The inequalities in (\ref{eqn:sequence_f}) yield $f(\mathbf{G}_i, \mathbf{c}_i) \geq f(\mathbf{G}_{i+1}, \mathbf{c}_{i+1}) \,\,\forall i$ which, in addition to the continuity of $f(\mathbf{G}, \mathbf{c})$ and to the compactness of the sublevel set, implies that the sequence $\{ f(\mathbf{G}_i, \mathbf{c}_i\}_{i \in \mathbb{N}}$ is non-increasing and thus has at least one limit point, $f(\bar{\mathbf{G}}, \bar{\mathbf{c}})$. Since the iterates also belong to a compact set, the generated sequence also admits some limit point, denoted as $(\bar{\mathbf{G}}, \bar{\mathbf{c}})$. 

    Now, we prove that $\{ \mathbf{G}_i \}_{i \in \mathbb{N}}$ converges to the aforementioned limit point. Consider a subsequence $\{ \mathbf{G}_{i_{k}}, \mathbf{c}_{i_{k}} \}_{k \in \mathbb{N}}$ that converges to the limit point $(\bar{\mathbf{G}}, \bar{\mathbf{c}})$. Without loss of generality, we can assume that the Grassmann variable, $\mathbf{G}_{i_{k}}$, is updated infinitely often in $\{ \mathbf{G}_{i_{k}}, \mathbf{c}_{i_{k}} \}_{k \in \mathbb{N}}$, so its convergence can be proved by contradiction. Let us assume that the Grassmann variable does not converge and hence there exists a positive value $\bar{\gamma}$ such that:
    \begin{equation}
        \label{eqn:distance_upper}
         d_c(\mathbf{G}_{i_{k+1}}, \mathbf{G}_{i_{k}}) = \gamma_{i_{k+1}} \geq \bar{\gamma} > 0.
    \end{equation}

    Also, consider that the aligned geodesic joining $\mathbf{G}_{i_{k+1}}$ and $\mathbf{G}_{i_{k}}$ with arclength $\gamma_{i_{k+1}}$ is given by:
    \begin{equation}
        \label{eqn:geodesic_iterations}
        \boldsymbol{\Gamma}_{i_{k+1}}(t) = \mathbf{G}_{a, i_{k}}\cos(\boldsymbol{\Theta}_{i_{k}} t) + \boldsymbol{\Delta}_{a, i_{k}}\sin(\boldsymbol{\Theta}_{i_{k}} t),
    \end{equation}
    where $\boldsymbol{\Delta}_{a, i_{k}}$ and $\boldsymbol{\Theta}_{i_{k}}$ are such that $\boldsymbol{\Gamma}_{i_{k+1}}(1) = \mathbf{G}_{a, i_{k+1}}$. With this geodesic in mind, we obtain:
    \begin{subequations}
        \label{eqn:inequalities_assumption}
        \begin{equation}
            f(\mathbf{G}_{i_{k+1}}, \mathbf{c}_{i_{k+1}}) \underbrace{\leq}_{(A2)} g_G(\mathbf{G}_{i_{k+1}} | \mathbf{G}_{i_{k}}, \mathbf{c}_{i_{k}}) \underbrace{=}_{\mathbf{G}_{i_{k+1}} = \boldsymbol{\Gamma}_{i_{k+1}}(1)}
        \end{equation}
        \begin{equation}
            g_G(\boldsymbol{\Gamma}_{i_{k+1}}(1) | \mathbf{G}_{i_{k}}, \mathbf{c}_{i_{k}}) \underbrace{\leq}_{\text{Eq. }(\ref{eqn:update_G})} g_G(\boldsymbol{\Gamma}_{i_{k+1}}(t) | \mathbf{G}_{i_{k}}, \mathbf{c}_{i_{k}}) \underbrace{=}_{\text{Eq. }(\ref{eqn:geodesic_iterations})}
        \end{equation}
        \begin{equation}
            \label{eqn:quasiconvexity_proof_1}
            g_G( \mathbf{G}_{a, i_{k}} \cos (\boldsymbol{\Theta}_{i_{k}} t) + \boldsymbol{\Delta}_{a, i_{k}} \sin (\boldsymbol{\Theta}_{i_{k}} t) | \mathbf{G}_{i_{k}}, \mathbf{c}_{i_{k}}) \underbrace{\leq}_{(A5)}
        \end{equation}
        \begin{equation}
            \label{eqn:quasiconvexity_proof_2}
            g_G(\mathbf{G}_{i_{k}} | \mathbf{G}_{i_{k}}, \mathbf{c}_{i_{k}}) \underbrace{=}_{(A1)} f(\mathbf{G}_{i_{k}}, \mathbf{c}_{i_{k}}),
        \end{equation}
    \end{subequations}
    from where we get:
    \begin{equation}
        \label{eqn:sandwich}
        f(\mathbf{G}_{i_{k+1}}, \mathbf{c}_{i_{k+1}})\leq g_G(\boldsymbol{\Gamma}_{i_{k+1}}(t) | \mathbf{G}_{i_{k}}, \mathbf{c}_{i_{k}}) \leq f(\mathbf{G}_{i_{k}}, \mathbf{c}_{i_{k}}).
    \end{equation}
    
    Considering that $\boldsymbol{\Theta}_{i_{k}}$ (because the principal angles belong to $[0, \frac{\pi}{2}]$) and $\boldsymbol{\Delta}_{a, i_{k}}$ (due to the orthonormality constraints, see (\ref{eqn:geo_aligned})) belong to closed and bounded sets, the sequences that they generate have limit points $\bar{\boldsymbol{\Theta}}$ and $\bar{\boldsymbol{\Delta}}_a$, respectively. As a consequence, the geodesic defined by those limit points is denoted as $\bar{\boldsymbol{\Gamma}}(t)$ and, with some abuse of notation, can be referred to as the \textit{limit point} geodesic. By further restricting to a subsequence that has limit points $\bar{\boldsymbol{\Theta}}$ and $\bar{\boldsymbol{\Delta}}_a$, invoking (A4) and letting $k \xrightarrow{} \infty$, (\ref{eqn:sandwich}) yields:
    \begin{equation}
        f(\bar{\mathbf{G}}, \bar{\mathbf{c}}) \leq g_G( \bar{\boldsymbol{\Gamma}}(t) | \bar{\mathbf{G}}, \bar{\mathbf{c}}) \leq f(\bar{\mathbf{G}}, \bar{\mathbf{c}}),
    \end{equation}
    which is equivalent to:
    \begin{equation}
        \label{eqn:limit_equality}
        \begin{gathered}
            f(\bar{\mathbf{G}}, \bar{\mathbf{c}}) = g_G( \bar{\boldsymbol{\Gamma}}(t) | \bar{\mathbf{G}}, \bar{\mathbf{c}}) = \\
            g_G( \bar{\mathbf{G}}_a\cos (\bar{\boldsymbol{\Theta}}t) + \bar{\boldsymbol{\Delta}}_a \sin (\bar{\boldsymbol{\Theta}} t) | \bar{\mathbf{G}}, \bar{\mathbf{c}}).
        \end{gathered}
    \end{equation}

    However, (\ref{eqn:limit_equality}) is contradictory with the unique minimizer assumption when the arclength of the geodesic, $\bar{\gamma}$, is different from $0$. From (A2), we know that:
    \begin{equation}
        g_G(\mathbf{G}_{i_{k+1}} | \mathbf{G}_{i_{k+1}}, \mathbf{c}_{i_{k+1}}) \leq g_G(\mathbf{G} | \mathbf{G}_{i_{k}}, \mathbf{c}_{i_{k}}) \,\,\,\,\forall \,\,\mathbf{G} \in \mathcal{G},
    \end{equation}
    whose limit for $k\xrightarrow{} \infty$ is:
    \begin{equation}
        \label{eqn:limit_upper_bound}
          g_G(\bar{\mathbf{G}} | \bar{\mathbf{G}}, \bar{\mathbf{c}}) \leq g_G(\mathbf{G} | \bar{\mathbf{G}}, \bar{\mathbf{c}}) \,\,\,\,\forall \,\,\mathbf{G} \in \mathcal{G},
    \end{equation}
    so $\bar{\mathbf{G}}$ is a minimizer of $g_G(\cdot | \bar{\mathbf{G}}, \bar{\mathbf{c}})$ and so is $\bar{\mathbf{G}}_a$. Since we assumed that the minimizers of the majorants are unique, this means that the limit point geodesic in (\ref{eqn:limit_equality}) must remain in the same point. As a result, (\ref{eqn:limit_equality}) is only true for $\bar{\boldsymbol{\Gamma}}(t) = \bar{\mathbf{G}}_a \,\,\forall t$ and, given that $\mathcal{G}$ is a geodesically convex subset (there is only a unique path joining two points), it means that:
    \begin{equation}
        \label{eqn:conv_cond}
         \lim_{k\xrightarrow{} \infty} d_c(\mathbf{G}_{i_{k+1}}, \mathbf{G}_{i_{k}}) = 0.
    \end{equation}
    
    Considering that the above argument must be satisfied by every subsequence, the sequence $\{ \mathbf{G}_i\}_{i \in \mathbb{N}}$ converges to $\bar{\mathbf{G}}$. Note that this limit point depends on the initialization points, $\mathbf{G}_0$ and $\mathbf{c}_0$. Likewise, the methodology that proves the convergence of the sequence $\{ \mathbf{c}_i \}_{i \in \mathbb{N}}$ is shown in  \cite[Theorem 2]{unified_convergence_analysis} and thus the joint sequence, $\{ \mathbf{G}_i, \mathbf{c}_i\}_{i \in \mathbb{N}}$, also converges.
    
    Finally, we prove that the limit point of the Grassmann variable iterates is a stationary point of the original problem. Notice that (\ref{eqn:limit_upper_bound}) implies, after taking the lower directional derivative from both sides at $\bar{\mathbf{G}}$, that:
    \begin{equation}
        \label{eqn:limit_directional_derivatives}
        g'_G(\bar{\mathbf{G}} | \bar{\mathbf{G}}, \bar{\mathbf{c}} ; \boldsymbol{\Delta}) \geq 0,
    \end{equation}
    for all directions $\boldsymbol{\Delta} \in \mathcal{T}_{\bar{\mathbf{G}}}\text{Gr}(N, D)$ whose respective geodesic emanating from $\bar{\mathbf{G}}$ stays in $\mathcal{G}$. Due to assumption (A3), $g'_G(\bar{\mathbf{G}} | \bar{\mathbf{G}}, \bar{\mathbf{c}} ; \boldsymbol{\Delta}) = f'(\bar{\mathbf{G}}, \bar{\mathbf{c}}; \boldsymbol{\Delta}, \mathbf{0})$, and thus:
    \begin{equation}
        f'(\bar{\mathbf{G}}, \bar{\mathbf{c}}; \boldsymbol{\Delta}, \mathbf{0}) \geq 0.
    \end{equation}

    In other words, $\bar{\mathbf{G}}$ is a coordinate-wise minimum of $f(\mathbf{G}, \mathbf{c})$. A similar argument is derived for the remaining convex block, $\mathbf{c}$, in \cite[Theorem 2]{unified_convergence_analysis} with the given assumptions. Since $\bar{\mathbf{G}}$ and $\bar{\mathbf{c}}$ are coordinate-wise minimums, this limit point is also a stationary point of $f(\mathbf{G}, \mathbf{c})$ thanks to the regularity of $f(\mathbf{G}, \mathbf{c})$.
\end{proof}
\begin{corollary}
    The optimality of the convergent points generated by the update equations in (\ref{eqn:MM_update_general}) depends on the cardinality of the set of stationary points of (\ref{eqn:prob_stat}). If the original problem has only one stationary point, then the previous theorem ensures that the convergence point is its (unique) global optimum. Otherwise, it converges to a locally optimal point.
\end{corollary}
\begin{example}
    \label{ex:blind_deconvo}
    An example of a block MM-like algorithm is found in the blind sparse deconvolution problem, which is based on the following optimization \cite[Eq. (6)]{global_geometry_sphere_constrained_sparse} \cite[Eq. (2.6)]{short_sparse_deconvolution}:
    \begin{equation}
        \label{eqn:blind_deconvo_problem}
        \min_{\mathbf{a}, \mathbf{x}} ||\mathbf{y} - \mathbf{a} \circledast \mathbf{x}||^2_2 + \lambda || \mathbf{x} ||_1 \,\,\,\,\text{s.t.}\,\,\mathbf{a} \in \text{Gr}(N, 1),
    \end{equation}
    where $\mathbf{a}$ is the convolution kernel, $\mathbf{x}$ is the estimated signal, $\circledast$ denotes the circular convolution and $\lambda$ is a regularizing parameter that is related to the sparsity of the solution. Although the approaches proposed in \cite{global_geometry_sphere_constrained_sparse, short_sparse_deconvolution} consider the Stiefel manifold, its convergence can be verified from Theorem \ref{theo:convergence} since there is a sign ambiguity \cite[Section 2.1]{global_geometry_sphere_constrained_sparse} and thus, \eqref{eqn:blind_deconvo_problem} satisfies the homogeneity condition of the Grassmannian \cite{orth_grassmann}, implying that it can also be constrained in terms of $\text{Gr}(N, 1)$. An intuitive approach is to optimize with respect to each block of variables in \eqref{eqn:blind_deconvo_problem} cyclically \cite{global_geometry_sphere_constrained_sparse, short_sparse_deconvolution}. In fact, the approach proposed in \cite{short_sparse_deconvolution} solves \eqref{eqn:blind_deconvo_problem} by proximal gradient methods and Riemannian gradient descent algorithms for $\mathbf{x}$ and $\mathbf{a}$, respectively. Both algorithms are known to be particular instances of the MM framework \cite{mm_review_palomar} and, as a by-product, the convergence of this approach is certified by Theorem \ref{theo:convergence}. For simulation results on this approach, see \cite[Sections 5 and 6]{short_sparse_deconvolution}.
\end{example}

%% file: Sections/conclusions.tex
In this letter, we have assessed the conditions in which a block MM algorithm converges when one of its blocks is constrained in the Grassmann manifold. Indeed, some of the classical assumptions about the majorant functions that are needed for the convergence of block MM algorithms are generalized within the geodesically convex optimization paradigm. The convergence theorem opens new geometric insights on known signal processing problems such as the one shown in Example \ref{ex:blind_deconvo}. Future lines of work aim to generalize Theorem \ref{theo:convergence} to other Riemannian manifolds and to study rigorously the convergence rate of block MM algorithms of this kind.